\documentclass{article}
\usepackage{amssymb,amsmath,amsthm,graphicx}
\usepackage[usenames,dvipsnames,svgnames,table]{xcolor}
\usepackage{color}
\newtheorem{theorem}{Theorem}

\newtheorem{proposition}[theorem]{Proposition}
\newtheorem{corollary}[theorem]{Corollary}
\newtheorem{conjecture}{Conjecture}
\theoremstyle{definition}
\newtheorem{example}{Example}
\newtheorem{definition}{Definition}

\date{}

\newcommand{\vc}[1]{\left\langle #1 \right\rangle}

\title{\Large \textbf{Boltzmann Enhancements of Biquasile Counting Invariants}}

\author{WonHyuk Choi\footnote{Email: wonhyuk.h.choi@gmail.com }\and
Deanna Needell \footnote{Email: deanna@math.ucla.edu. Partially supported by NSF CAREER $\#1348721$.} \and
Sam Nelson\footnote{Email: Sam.Nelson@cmc.edu. Partially supported by Simons Foundation collaboration grant $\#316709$.}}

\begin{document}
\maketitle

\begin{abstract}
In this paper, we build on the biquasiles and dual graph diagrams introduced 
in \cite{dsn2}. We introduce \textit{biquasile Boltzmann weights} that enhance 
the previous knot coloring invariant defined in terms of finite biquasiles and 
provide examples differentiating links with the same counting invariant, 
demonstrating that the enhancement is proper. We identify conditions for
a linear function $\phi:\mathbb{Z}_n[X^3]\to\mathbb{Z}_n$ to be a Boltzmann
weight for an Alexander biquasile $X$.  
\end{abstract}

\parbox{5.5in} {\textsc{Keywords:} biquasiles, dual graph diagrams, 
enhancements of counting invariants, Boltzmann weights

\smallskip

}

\section{Introduction}

In \cite{dsn2}, the second and third listed authors introduced a 
combinatorial structure known as \textit{dual graph diagrams} for 
representing oriented knots and links and a corresponding algebraic
structure known as \textit{biquasiles} for 
defining knot and link invariants via counting vertex colorings of
dual graph diagrams with biquasile elements satisfying certain conditions.
Dual graph diagrams arise from taking both checkerboard graphs of a
knot or link diagram together, sometimes called the \textit{overlaid 
Tait graph}, and adding edge decorations to indicate crossing signs and 
orientations. Biquasiles are algebraic structures consisting of two 
quasigroup operations on a set $X$ which interact according to certain
identities, analogous in some sense to the two group structures on a field
interacting via the distributive law.

Given an oriented knot or link $L$ and a finite biquasile $X$, the number of
vertex colorings of the corresponding dual graph diagram by $X$ is unchanged
by Reidemeister moves and hence defines an integer valued computable invariant
of oriented knots and links. Starting with \cite{CJKLS} and continuing with 
subsequent papers such as \cite{CES, CNS,NN} and many more, counting
invariants of knotted objects associated to various algebraic structures 
such as quandles, biquandles, racks, biracks, kei and bikei
have been enhanced to obtain stronger invariants by defining
invariants $\phi$ of colored knots or links. The resulting multiset of 
$\phi$-values over the set of colorings of a knot or link then defines a
generally stronger invariant whose cardinality recovers the original counting
invariant.

In this paper we enhance the biquasile counting invariant with 
\textit{Boltzmann weights}, functions from the set of ordered triples of
elements of a biquasile $X$ to an abelian group $A$ with the property that
the sum of Boltzmann weights at crossings is unchanged by Reidemeister moves
and hence defines an $A$-valued invariant of $X$-colored dual graph diagrams 
under Reidemeister equivalence, analogously to the quandle and biquandle 
2-cocycle invariants studied in \cite{CJKLS,EN} etc. The paper is organized 
as follows. In Section \ref{DB} we recall the basics of dual graph diagrams
and biquasiles. In Section \ref{B} we define Boltzmann enhancements and 
provide examples.
We close in Section \ref{Q} with some questions for future research.

\section{Dual Graph Diagrams and Biquasiles}\label{DB}

We begin with two notions from \cite{dsn2}.
\begin{definition}
Let $D$ be an oriented knot or link diagram. We form the \textit{dual graph
diagram} associated to $D$ by placing a vertex in every region and making 
two regions adjacent if they are opposite at a crossing. We then give each edge 
either a direction or a $+$ or $-$ sign as depicted:
\[\includegraphics{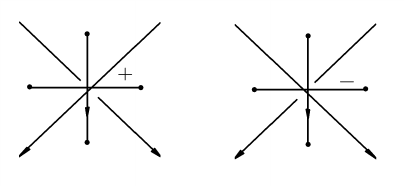}\]
\end{definition}

\begin{example}
The Hopf link below has the pictured dual graph diagram.
\[\includegraphics{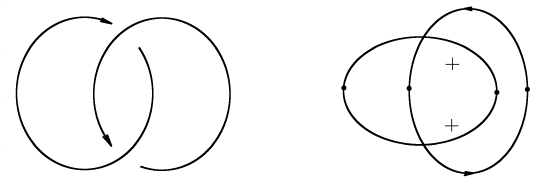}\]
\end{example}

As noted in \cite{dsn2}, dual graph diagrams are always pairs of dual planar
graphs with the property that each edge crosses exactly one other edges, with 
crossing edges forming pairs where one edge has a sign and the other has a 
direction. 

\begin{definition}
\label{dual_graph_reid} 
Two dual graph diagrams are \textit{equivalent} if they are related by a 
sequence of the following moves:
\[\includegraphics{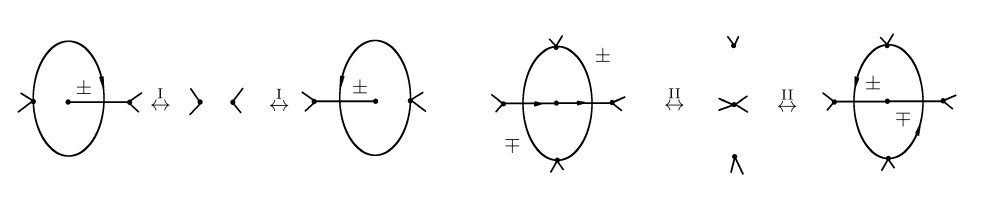}\]
\[\includegraphics{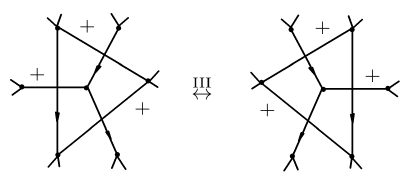}\]
\end{definition}

These dual graph moves form a generating set of the oriented Reidemeister move expressed in a dual graph format. Next, we define biquasiles. 

\begin{definition}
Let $X$ be a set with binary operations $\ast, \cdot, \backslash^{\ast}, /^{\ast}, 
\backslash, /:X\times X\to X$ satisfying
\[\begin{array}{rcccl}
y\backslash^{\ast}(y\ast x) & = & x & = & (x\ast y)/^{\ast} y \\
y\backslash (y\cdot x) & = & x & = & (x\cdot y)/y. 
\end{array}\]
Then we say $X$ is a \textit{biquasile} if for all $a,b,x,y\in X$ we have
\[\begin{array}{rcll}
a\ast(x\cdot [y\ast(a\cdot b)]) & = & (a\ast[x\cdot y])\ast(x\cdot [y\ast([a\ast(x\cdot y)]\cdot b)]) & (i) \\
y\ast([a\ast (x\cdot y)]\cdot b) & = & (y\ast[a\cdot b])\ast([a\ast (x\cdot [y\ast(a\cdot b)])]\cdot b) & (ii).
\end{array}\]
\end{definition}

\begin{example}
\label{Alexander}
Let $R$ be any commutative ring with identity and let $d,s,n\in R$ be units.
Then $X$ is a biquasile under the operations
\[x\cdot y= dx+sy\quad\mathrm{and}\quad x\ast y=-dsn^2x+ny.\]
Such a biquasile is called an \textit{Alexander biquasile}; see \cite{dsn2} 
for more.
\end{example}
\begin{example}\label{ex:bq1}
For any finite set $X=\{x_1,\dots, x_n\}$ we can define a biquasile structure on
$X$ by choosing operation tables for $\ast$ and $\cdot$ so that the biquasile 
axioms are satisfied. To save writing, we can drop the ``$x$''s and write only
the subscripts, resulting in a block matrix. For example, the Alexander 
biquasile structure on $\mathbb{Z}_3=\{x_1=1,x_2=2,x_3=3\}$ (where we use 3 for 
the class of zero in $\mathbb{Z}_3$ since we start numbering our rows and 
columns with 1 instead of 0) with $d=1$, $s=1$ and $n=2$ has operations 
$x\ast y=-dsn^2+ny=2x+2y$ and $x\cdot y=dx+sy=x+y$ with operation tables and matrix
\[
\begin{array}{r|rrr}
\ast & 1 & 2 & 3 \\ \hline
1 & 1 & 3 & 2\\
2 & 3 & 2 & 1\\
3 & 2 & 1 & 3
\end{array}
\quad
\begin{array}{r|rrr} 
\cdot & 1 & 2 & 3 \\ \hline
1 & 2 & 3 & 1\\
2 & 3 & 1 & 2\\
3 & 1 & 2 & 3
\end{array}\quad
\leftrightarrow\quad
\left[\begin{array}{rrr|rrr}
1 & 3 & 2 &2 & 3 & 1 \\
3 & 2 & 1 &3 & 1 & 2 \\
2 & 1 & 3 &1 & 2 & 3 \\
\end{array}\right].
\]
\end{example}

\begin{definition}
Given a dual graph diagram $D$ and a biquasile $X$, an \textit{$X$-coloring}
of $D$ is an assignment of elements of $X$ to the vertices of $D$ such that 
at every crossing we have the following pictures:
\[\includegraphics{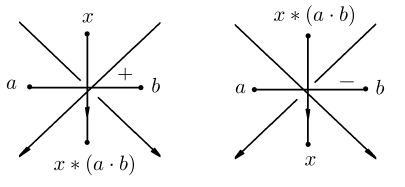}\]
\end{definition}

The biquasile axioms are chosen so that given a valid $X$-coloring of a diagram
before a move, there is a unique corresponding valid coloring of the resulting
diagram after the move. It follows that the number of $X$-colorings of
a dual graph diagram is an invariant of oriented knots and links, denoted
$\Phi_X^{\mathbb{Z}}(L)$, called the \textit{biquasile counting invariant}. 

\begin{definition}
Let $L$ be a dual graph diagram and $X$ the set of its vertices. Then the 
\textit{fundamental biquasile} of $L$ is the biquasile with presentation
 $\vc{\ X \mid R\ }$ where for each edge crossing, a relation is defined by the figure above. 
More precisely,
the elements of the fundamental biquasile are equivalence classes of biquasile 
words in generators corresponding to the vertices of the dual graph diagram 
modulo the equivalence relation generated by the crossing relations and biquasile axioms. See \cite{dsn2} for more.
\end{definition}

\begin{example}
Let us assign generators to the vertices in the Hopf link dual graph as 
pictured.
\[\includegraphics{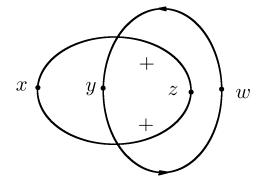}\]
Reading the crossing relations from the diagram, the Hopf link has
fundamental biquasile presentation
\[\mathcal{FB}(L)=\langle x,y,z,w\ |\ y=w\ast(x\cdot z), w=y\ast(x\cdot z)\rangle.\]

Then with coloring biquasile $X$ given by the Alexander biquasile 
$\mathbb{Z}_3$ with $d=s=1$ and $n=2$,
we obtain coloring equations 
\begin{eqnarray*}
y & = & w\ast(x\cdot z)\\
 & = & x+z+ 2w\quad \mathrm{and} \\
w & = & y\ast(x\cdot z)\\
 & = & x+z+2y\\
\end{eqnarray*}
so we have a homogeneous system of linear 
equations over $\mathbb{Z}_3$ with coefficient matrix
\[\left[\begin{array}{rrrr}
1 & 2 & 1 & 2 \\
1 & 2 & 1 & 2
\end{array}\right].\]
After row reduction over $\mathbb{Z}_3$, we obtain
\[
\left[\begin{array}{rrrr}
1 & 2 & 1 & 2 \\
0 & 0 & 0 & 0
\end{array}\right]
\]
so the kernel has dimension 3 and we have $\Phi_X^{\mathbb{Z}}(L)=3^3=27.$
\end{example}

\section{Boltzmann Enhancements}\label{B}

In this section we will enhance the biquasile counting invariant using Boltzmann weights valued
in an abelian group $A$,  a 
strategy which has proved effective in the cases of other knot coloring
structures such as quandles and biquandles.

\begin{definition}
\label{BW_Enhancement}
Let $X$ be a biquasile and $A$ an abelian group. Then a \textit{biquasile Boltzmann 
weight} is an $A$-linear map $\phi: A[X^3] \to A$ such that for all 
$x,y,a,b\in X$ we have
\begin{itemize}
\item[(i)] \[
\phi(x,a,a\backslash(x\backslash^{\ast}x))
=\phi(x,(x\backslash^{\ast} x)/b,b)
=0
\]
and
\item[(ii)]
\[\begin{array}{l}
\phi(x,a,b)+\phi(b,x\ast(a\cdot b),y)
+\phi(x\ast(a\cdot b),a,b\ast([x\ast(a\cdot b)]\cdot y)) \\
 = 
\phi(b,x,y) +
\phi(x,a,b\ast(x\cdot y))
+\phi(b\ast(x\cdot y), x\ast(a\cdot[b\ast(x\cdot y)]),y).
\end{array}\]
\end{itemize}
\end{definition}

The biquasile Boltzmann weight axioms are chosen so that the sum of Boltzmann 
weight values according to the rules
\[\includegraphics{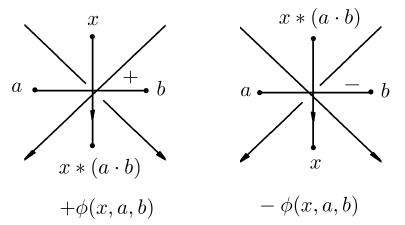}\]
are unchanged by dual graph Reidemeister moves  as expressed in Definition 
\ref{dual_graph_reid}. We prove this in the following 
proposition: 

\begin{proposition}
Let $D$ be a dual graph diagram with a coloring by a biquasile $X$. Then
if $\phi:X^3\to A$ is a biquasile Boltzmann weight, the sum of the $\phi$
values over all edge crossings in $D$ is unchanged by dual graph
Reidemeister moves. 
\end{proposition}

\begin{proof}
Comparing the two sides of the dual graph Reidemeister I moves as 
labeled, 
\[\includegraphics{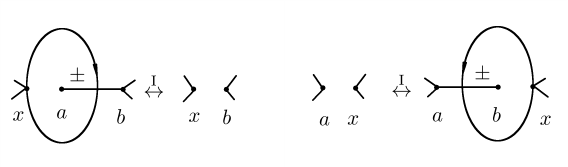}\]
we have the requirement that 
\[\pm \phi(x,a,b)=0\]
when $x=x\ast(a\cdot b)$. Solving for $a$ and $b$, this
yields the requirements that
\[\pm\phi(x,a,a\backslash(x\backslash^{\ast}x))
=\pm\phi(x,(x\backslash^{\ast} x)/b,b)=0\]
for all $x,a,b\in X$.

Our choice of $\phi(x,a,b)$ at positive crossings and $-\phi(x,a,b)$ at
negative crossings (with $x$ the output label) satisfies the Reidemeister II
moves. More precisely,
comparing the two sides of the dual graph Reidemeister II moves as 
labeled yields $\phi(x,a,b)-\phi(x,a,b)$ and $0$ respectively:
\[\includegraphics{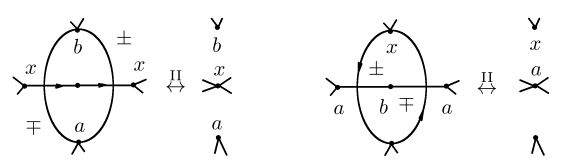}\]

Finally, comparing the two sides of the dual graph Reidemeister III move as 
labeled yields
\[\phi(x,a,b)+\phi(b,x\ast(a\cdot b),y)
+\phi(x\ast(a\cdot b),a,b\ast([x\ast(a\cdot b)]\cdot y))\] on one side
and  
\[\phi(b,x,y) +
\phi(x,a,b\ast(x\cdot y))
+\phi(b\ast(x\cdot y), x\ast(a\cdot[b\ast(x\cdot y)]),y)\]
on the other:
\[\includegraphics{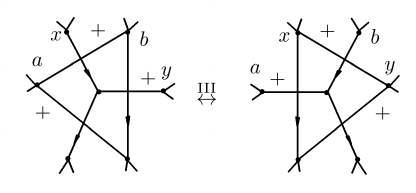}\]
\end{proof}

\begin{corollary}
If $X$ is a biquasile and $\phi:A[X^3]\to A$ is a biquasile Boltzmann weight,
then the multiset $\Phi_X^{M,\phi}(L)$ of $\phi$ values over the set of 
$X$-colorings of a dual graph diagram $D$ representing an oriented knot or 
link $L$ is an invariant of knots and links.
\end{corollary}

We can convert the multiset of Boltzmann weights into a ``polynomial'' form by
converting multiplicities to coefficients and elements to exponents of a formal 
variable $u$ for ease of comparison, e.g. the multiset $\{0,0,0,1,1,2,3,3\}$ 
becomes $3+2u+u^2+2u^3$. With this notation, we will write
\[\Phi_X^{\phi}(L)=\sum_{f \in \{X-\mathrm{colorings}\}} u^{BW(f)}\]
and call $\Phi^{\phi}_X(L)$ the \textit{Boltzmann enhanced polynomial} of $L$
with respect to the biquasile $X$ and the Boltzmann weight $\phi$. 
While this convention only yields a true (Laurent) polynomial 
if the abelian group $A$ is the integers $\mathbb{Z}$, this notation has the 
advantage that evaluation at $u=1$ (using the rule that $1^x=1$ for all 
$x\in A$) recovers the cardinality of the multiset.

As an $A$-linear function, a Boltzmann weight can be conveniently expressed
as an $A$-linear combination of characteristic functions 
\[\chi_{x,y,z}(a,b,c)=\left\{\begin{array}{ll}
1 & (a,b,c)=(x,y,z) \\
0 & (a,b,c)\ne(x,y,z). \\
\end{array}\right.
\]

\begin{example}\label{ex:aff}
Let $X$ be the biquasile with operation matrix
\[\left[\begin{array}{rr|rr}
1 & 2 & 2 & 1\\
2 & 1 & 1 & 2
\end{array}\right].\]
This biquasile can be written as $\mathbb{Z}_2=\{1,2\}$ with operations
$x\ast y=x+y$ and $x\cdot y=x+y+1$.
Then our \texttt{Python} computations reveal 125 Boltzmann weights on $X$ with values in
$\mathbb{Z}_5$, including for instance 
\[\phi=2\chi_{1,1,1}+3\chi_{1,2,2}+4\chi_{2,1,1}+3\chi_{2,2,2}.\]
The Hopf link $L2a1$ and the $(4,2)$-torus link $L4a1$ have the same counting 
invariant value $\Phi_X^{\mathbb{Z}}(L2a1)=\Phi_X^{\mathbb{Z}}(L4a1)=8$, 
distinguishing both from the unlink of two components which has counting 
invariant value $\Phi_X^{\mathbb{Z}}(U_2)=4$. Let us use the Boltzmann enhancement
to distinguish the two.

The Hopf link has eight biquasile colorings by $X$, including for instance
\[\includegraphics{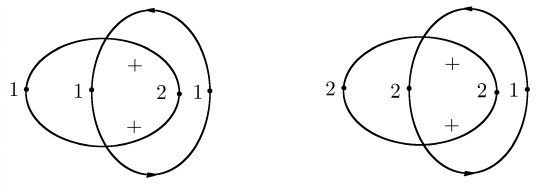}\]
The coloring on the left has Boltzmann weight $2\phi(1,1,2)=2(0)=0$
while the coloring on the right has Boltzmann weight
$\phi(2,2,2)+\phi(1,2,2)=3+3=1$. Computing the other six Boltzmann weights,
we obtain $\Phi^{\phi}_X(L)=4+4u$.

The $(4,2)$-torus link has eight $X$-colorings including
\[\includegraphics{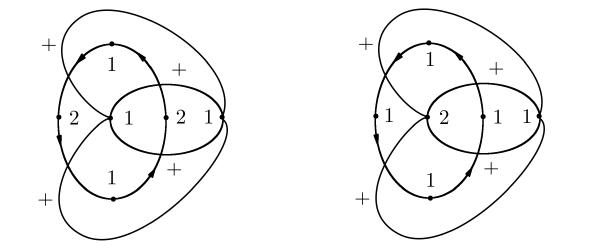}\]
The coloring on the left has Boltzmann weight $2\phi(1,1,1)+2\phi(2,1,1)=4+8=2$
while the coloring on the right has Boltzmann weight
$4\phi(1,2,1)=0$. Computing the other six Boltzmann weights,
we obtain $\Phi^{\phi}_X(L)=4+4u^2$, which distinguishes this link from the Hopf 
link. In particular, this example shows that the Boltzmann enhancement is a 
proper enhancement, i.e., a stronger invariant than the unenhanced biquasile counting
invariant.
\end{example}

\begin{example}
Continuing with the biquasile $X$ from example \ref{ex:aff}, we selected three
Boltzmann weights with values in $\mathbb{Z}_6$ and computed $\Phi^{\phi}_X(L)$ for the list of prime links
with up to seven crossings as listed in the Knot Atlas \cite{KA}; the results
are collected in the table.
\[\begin{array}{rcl}
\phi_1 & = & \chi_{1,1,1}+5\chi_{1,2,2}+3\chi_{2,1,1}+5\chi_{2,2,2}, \\ 
\phi_2 & = & \chi_{1,1,1}+\chi_{1,2,2}+2\chi_{2,1,1}+2\chi_{2,2,2}\quad \mathrm{and}\\
\phi_3 & = & 4\chi_{1,2,2}+2\chi_{2,2,2}
\end{array}\]
\[
\begin{array}{|c|ccccccccc|}\hline
L & L2a1 & L4a1 & L5a1 & L6a1 & L6a2 & L6a3 & L6a4 & L6a5 & L61 \\ \hline
\Phi_X^{\phi_1}(L) &4+4u^4 & 4+4u^2 & 8 & 4+4u^2 & 8 & 8 & 4 & 
4+12u^2 & 4+12u^2 \\
\Phi_X^{\phi_2}(L) & 4+4u^3 & 8 & 8 & 8 & 4+4u^3 & 4+4u^3 & 4 & 
16 & 16 \\
\Phi_X^{\phi_3}(L) & 8 & 8 & 8 & 8 & 8 & 8 & 4 & 16 & 16 \\
\hline
L & L7a1 & L7a2 & L7a3 & L7a4 & L7a5 & L7a6 & L7a7 & L7n1 & L7n2 \\\hline
\Phi_X^{\phi_1}(L) & 8& 4+4u^2& 8 & 8 & 4+4u^4 & 4+4u^4 & 12+4u^2 & 4+4u^4 & 8 \\
\Phi_X^{\phi_2}(L) & 8 & 8 & 8& 8 & 4+4u^3 & 4+4u^3 & 16 & 8 & 8\\
\Phi_X^{\phi_3}(L) & 8 & 8 & 8 & 8 & 8 & 8 & 16 & 8 & 8\\
\hline
\end{array}
\]
We observe that the $\phi_3$ weight yields just the biquasile counting 
invariant for the links in the table, while $\phi_1$ and $\phi_2$ both yield 
proper enhancements.
\end{example}

\begin{proposition} 
\label{linear}
Let $A=\mathbb{Z}_n$ and $X=\mathbb{Z}_n$ with a choice of 
$d,s,n\in X^{\times}$, making $X$
a finite Alexander biquasile. Then for any $\gamma\in\mathbb{Z}_n$, the map 
$\phi:\mathbb{Z}_n[X]^3\to\mathbb{Z}_n$ given by
\[\phi(x,y,z) = -\gamma(s^{-1}n^{-1} +dn) x + \gamma s^{-1}d y + \gamma z\]
defines a Boltzmann weight which we call a \textup{linear Boltzmann weight}. 
\end{proposition}

\begin{proof}
Let $\phi(x,y,z) = -\gamma(s^{-1}n^{-1} +dn) x + \gamma s^{-1}d y + \gamma z$.
Recall that our biquasile operations are given by
\[x\ast y=-dsn^2 x+n y\quad\mathrm{and}\quad x\cdot y=dx+sy.\]
Then observing that
\[\begin{array}{rcl}
x\backslash^{\ast} y & = & dsn x+n^{-1} y, \\
x\backslash y & = & -ds^{-1}x+s^{-1}y \ \mathrm{and} \\
x/y & = & d^{-1}x-d^{-1}sy,
\end{array}\]
we compute
\begin{eqnarray*}
\phi(x,a,a\backslash(x\backslash^{\ast}x))
& = & \phi(x,a,-ds^{-1}a+s^{-1}(dsn x+n^{-1} x))\\
& = & -\gamma(s^{-1}n^{-1} +dn) x + \gamma s^{-1}d a + \gamma (-ds^{-1}a+s^{-1}(dsn+n^{-1}) x)) \\
& = & (-\gamma s^{-1}n^{-1} -\gamma dn+\gamma dn+\gamma s^{-1}n^{-1})) x + (\gamma s^{-1}d -\gamma ds^{-1})a\\
& = & 0,
\end{eqnarray*}
\begin{eqnarray*}
\phi(x,(x\backslash^{\ast} x)/b,b)
& = & \phi(x,d^{-1}(dsn x+n^{-1} x)-d^{-1}sb,b)\\
& = &-\gamma(s^{-1}n^{-1} +dn) x + \gamma s^{-1}d (d^{-1}(dsn x+n^{-1} x)-d^{-1}sb) + \gamma b \\
& = & (-\gamma s^{-1}n^{-1} -\gamma dn+\gamma s^{-1}dd^{-1}dsn+ \gamma s^{-1}dd^{-1}n^{-1}) x + (-\gamma s^{-1}dd^{-1}s + \gamma) b \\
& = & (-\gamma s^{-1}n^{-1} -\gamma dn+\gamma dn+ \gamma s^{-1}n^{-1}) x + (-\gamma + \gamma) b \\
& = & 0
\end{eqnarray*}
so condition (i) is satisfied.

Checking condition (ii), we have on the left side
\begin{eqnarray*}
L& = & \phi(x,a,b)+\phi(b,x\ast(a\cdot b),y)
+\phi(x\ast(a\cdot b),a,b\ast([x\ast(a\cdot b)]\cdot y))\\
& = &
-\gamma(s^{-1}n^{-1} +dn) x + \gamma s^{-1}d a + \gamma b
-\gamma(s^{-1}n^{-1} +dn) b + \gamma s^{-1}d (-dsn^2 x+nd a+ns b) + \gamma y  \\ &&
-\gamma(s^{-1}n^{-1} +dn) (-dsn^2 x +nd a+ns b)+ \gamma s^{-1}d a 
+ \gamma (-dsn^2 b +n[d(-dsn^2 x +nd a+ns b)+sy]) \\
& = & 
\gamma ((-s^{-1}n^{-1}+dn) + s^{-1}d (-dsn^2)-(s^{-1}n^{-1} +dn)(-dsn^2) + nd(-dsn^2))x \\&&
+\gamma (s^{-1}d+(s^{-1}d)nd- (s^{-1}n^{-1} +dn)nd+s^{-1}d+nd(dn) )a \\
& & +\gamma (1+ns)y +\gamma (1-(s^{-1}n^{-1} +dn)+s^{-1}dns-(s^{-1}n^{-1} +dn)ns -dsn^2 +ndns)b\\
& = & 
\gamma (-s^{-1}n^{-1}-dn-d^2n^2 +dn +d^2sn^3 -d^2sn^3)x
+\gamma (s^{-1}d+s^{-1}nd^2- s^{-1}d -d^2n^2+s^{-1}d+d^2n^2 )a \\
& & +\gamma (1+ns)y 
+\gamma (1-s^{-1}n^{-1} -dn+dn -1 -dsn^2 -dsn^2 +dsn^2)b\\
& = & 
\gamma (-s^{-1}n^{-1}-d^2n^2)x
+\gamma (s^{-1}nd^2+s^{-1}d )a +\gamma (1+ns)y 
+\gamma (-s^{-1}n^{-1} -dsn^2)b
\end{eqnarray*}
while on the right side we have
\begin{eqnarray*}
R & = &\phi(b,x,y) +
\phi(x,a,b\ast(x\cdot y))
+\phi(b\ast(x\cdot y), x\ast(a\cdot[b\ast(x\cdot y)]),y) \\ 
& = &
-\gamma(s^{-1}n^{-1} +dn) b + \gamma s^{-1}d x + \gamma y
-\gamma(s^{-1}n^{-1} +dn) x + \gamma s^{-1}d a + \gamma (-dsn^2 b+nd x+ns y)\\ & &
-\gamma(s^{-1}n^{-1} +dn)(-dsn^2 b+nd x+ns y) 
+ \gamma s^{-1}d (-dsn^2 x+nd a+ns(-dsn^2 b+nd x+ns y) + \gamma y \\
& = & 
\gamma (s^{-1}d -(s^{-1}n^{-1} +dn)+nd -(s^{-1}n^{-1} +dn)nd+s^{-1}d(-dsn^2+nsnd) )x
\\&&
+\gamma (s^{-1}d+s^{-1}dnd)a
+\gamma (1+ns-(s^{-1}n^{-1} +dn)ns+ s^{-1}d(nsns)+1  )y\\&&
+\gamma (-(s^{-1}n^{-1} +dn)-dsn^2-(s^{-1}n^{-1} +dn)(-dsn^2)+s^{-1}dns(-dsn^2))b\\
& = & 
\gamma (s^{-1}d-s^{-1}n^{-1} -dn +nd -s^{-1}d -d^2n^2-d^2n^2+d^2n^2 )x\\&&
+\gamma (s^{-1}d+s^{-1}nd^2)a
+\gamma (1+ns-1 -dsn^2+ dsn^2+1  )y\\&&
+\gamma (-s^{-1}n^{-1} -dn-dsn^2+dn +d^2sn^3-d^2sn^3)b\\
& = & 
\gamma (-s^{-1}n^{-1} -d^2n^2 )x
+\gamma (s^{-1}d+s^{-1}nd^2)a
+\gamma (1+ns)y
+\gamma (-s^{-1}n^{-1}-dsn^2)b\\
\end{eqnarray*}
as required.
\end{proof}

\begin{example}\label{ex:lin}
The Alexander biquasile $\mathbb{Z}_3$ with $d=s=2$ and $n=1$ has operation 
matrix
\[\left[\begin{array}{rrr|rrr}
3 & 1 & 2 & 1 & 3 & 2 \\
2 & 3 & 1 & 3 & 2 & 1 \\
1 & 2 & 3 & 2 & 1 & 3
\end{array}\right].\]
Then up to scalar multiplication, it has one linear Boltzmann weight
$\phi:\mathbb{Z}_3[X^3]\to\mathbb{Z}_3$ given by
\begin{eqnarray*}
\phi & = & 
2\chi_{1,1,1}+\chi_{1,1,2}+\chi_{1,2,1}+2\chi_{1,2,3}+2\chi_{1,3,2}+\chi_{1,3,3} \\ & &
+2\chi_{2,1,2}+\chi_{2,3,1}+2\chi_{2,2,1}+\chi_{2,2,2}+\chi_{2,3,1}+2\chi_{2,3,3}\\ & &
+\chi_{3,1,1}+2\chi_{3,1,3}+2\chi_{3,2,2}+\chi_{3,2,3}+2\chi_{3,3,1}+\chi_{3,3,2}.
\end{eqnarray*}
Our \texttt{Python} computations say that this Boltzmann weight defines a 
trivial enhancement (i.e., just the counting invariant) for all prime
knots with up to 8 crossings and all prime links with up to 7 crossings. However, this is not the only trivial case. In fact, 
\end{example}

\begin{conjecture}
\label{conjecture}
The linear Boltzmann weights as defined by Proposition \ref{linear} is always a trivial enhancement. 
\end{conjecture}

Our \texttt{Python} computations reveal that all Alexander biquasiles over 
$\mathbb Z_n$ where  $0<n \leq 7$ only have trivial 
linear Boltzmann enhancements on prime classical knots with up to eight 
crossings and prime classical links with up to seven crossings.

\section{Questions}\label{Q}

We end with a few collected questions for future research.

\begin{itemize}
\item In several recent papers such as \cite{MN1, MN2}, Niebrzydowski and 
collaborators have studied region colorings of knots and link by structures 
called \textit{knot-theoretic ternary quasigroups} and have introduced a 
homology theory for these structures. A biquasile determines a knot-theoretic 
ternary quasigroup by 
\[abcT=b\ast(a\cdot c)\]
and our Boltzmann weights correspond to these cocycles. What is the exact 
relationship between these structures?

\item Given two distinct knots or links, is there always a pair of biquasile 
and Boltzmann weight such that the enhanced invariant distinguishes the pair?

\item Is Conjecture \ref{conjecture} true? 

\item What are the conditions for polynomial Boltzmann weights of higher degree
for Alexander biquasiles analogous to the linear Boltzmann weights in 
Proposition \ref{linear}? 

\end{itemize}

\bibliography{whc-dn-sn}{}
\bibliographystyle{abbrv}

\bigskip

\noindent
\textsc{Department of Mathematics \\
Univ. of California, Los Angeles \\
520 Portola Plaza, Los Angeles, CA 90095 }

\bigskip

\noindent
\textsc{Department of Mathematical Sciences \\
Claremont McKenna College \\
850 Columbia Ave. \\
Claremont, CA 91711} 

\medskip

\end{document}